\documentclass[a4paper]{amsart}

\usepackage[utf8]{inputenc}
\usepackage{amsmath}
\usepackage{amsfonts}
\usepackage{xypic}
\usepackage{amsthm}
\usepackage{mathrsfs}
\usepackage{amssymb}
\usepackage[all]{xy}
\xyoption{rotate}
\usepackage{color}
\usepackage{hyperref}
\usepackage{graphicx}
\usepackage{wrapfig}
\usepackage{float}
\usepackage{xspace}
\graphicspath{ {dibujos/} }

\numberwithin{equation}{section}

\newtheorem{theorem}{Theorem}[section]
\newtheorem{lemma}[theorem]{Lemma}
\newtheorem{proposition}[theorem]{Proposition}

\theoremstyle{remark}
\newtheorem{remark}[theorem]{Remark}

\newtheorem{innercustomgeneric}{\customgenericname}
\providecommand{\customgenericname}{}
\newcommand{\newcustomtheorem}[2]{%
  \newenvironment{#1}[1]
  {%
   \renewcommand\customgenericname{#2}%
   \renewcommand\theinnercustomgeneric{##1}%
   \innercustomgeneric
  }
  {\endinnercustomgeneric}
}

\newcustomtheorem{customthm}{Theorem}
\newcustomtheorem{customprop}{Proposition}
\newcustomtheorem{customlemma}{Lemma}

\newtheoremstyle{TheoremNum}
        {\topsep}{\topsep}              
        {\itshape}                      
        {}                              
        {\bfseries}                     
        {.}                             
        { }                             
        {\thmname{#1}\thmnote{ \bfseries #3}}
    \theoremstyle{TheoremNum}

\newcommand{\Pic}{\mathrm{Pic}}

\newcommand{\End}{\mathrm{End}}

\newcommand{\Aa}{\mathcal{A}}

\newcommand{\Dd}{\mathcal{D}}
\newcommand{\Ee}{\mathcal{E}}
\newcommand{\Ff}{\mathcal{F}}

\newcommand{\Ll}{\mathcal{L}}

\newcommand{\Vv}{\mathcal{V}}
\newcommand{\Ww}{\mathcal{W}}

\renewcommand{\H}{\mathrm{H}}

\newcommand{\M}{\mathrm{M}}

\newcommand{\lc}{\left \lceil}
\newcommand{\rc}{\right \rceil}

\newcommand{\CC}{\mathbb{C}}

\newcommand{\PP}{\mathbb{P}}

\newcommand\Quotient[2]{
        \mathchoice
            {
                \text{\raise1ex\hbox{\thinspace $#1$}\Big{/} \lower1ex\hbox{$#2$} \thinspace}%
            }
            {
                #1\,/\,#2
            }
            {
                #1\,/\,#2
            }
            {
                #1\,/\,#2
            }
    }

\newcommand\GIT[2]{
        \mathchoice
            {
                \text{\raise1ex\hbox{\thinspace $#1$}\Big{/}\!\!\!\!\Big{/} \lower1ex\hbox{$#2$} \thinspace}%
            }
            {
                #1\,/\,#2
            }
            {
                #1\,/\,#2
            }
            {
                #1\,/\,#2
     a       }
    }

\title{The wobbly divisors of the moduli space of rank-$2$ vector bundles}

\author{Sarbeswar Pal}

\address{Sarbeswar Pal \\ Indian Institute of Science Education and Research \\ Thiruvananthapuram\\
Maruthamala PO\\ Vithura\\
Thiruvananthapuram - 695551\\  Kerala\\  India }
\email{sarbeswar11@gmail.com, spal@iisertvm.ac.in}

\author{Christian Pauly}
\address{Christian Pauly \\ Laboratoire de Math\'ematiques J.A. Dieudonn\'e \\ UMR  7351 CNRS \\ Universit\'e de Nice 
Sophia-Antipolis \\ 06108 Nice Cedex 02, France}
\email{pauly@unice.fr}

\date{\today}

\begin{document}

\maketitle

\begin{abstract}
Let $X$ be a smooth projective complex curve of genus $g \geq 2$ and let $\M_X(2,\Lambda)$
be the moduli space of semi-stable rank-$2$ vector bundles over $X$ with fixed determinant $\Lambda$. 
We show that the wobbly locus, i.e., the locus of semi-stable vector bundles admitting a non-zero
nilpotent Higgs field is a union of divisors $\Ww_k \subset \M_X(2,\Lambda)$. We show that on
one wobbly divisor the set of maximal subbundles is degenerate. We also compute the class of the
divisors $\Ww_k$ in the Picard group of $\M_X(2,\Lambda)$.
\end{abstract}

\section{Introduction}

Let $X$ be a smooth projective complex curve of genus $g \geq 2$ and let $K$ be its canonical line bundle. 
Fixing a line bundle $\Lambda$ we consider the coarse moduli space 
$\M_X(2, \Lambda)$ parameterizing semi-stable rank-$2$ vector bundles of fixed
determinant $\Lambda$ over $X$. In this note we study the locus in $\M_X(2, \Lambda)$ of wobbly, or non-very stable, vector bundles over
$X$. We recall that a vector bundle $E$ is called very stable if $E$ has no non-zero nilpotent Higgs field
$\phi \in \H^0(X, \End(E) \otimes K)$. Laumon \cite[Proposition 3.5]{L} proved, assuming $g \geq 2$, that a 
very stable vector bundle is stable and that the locus of very stable bundles is a non-empty open subset of 
$\M_X(2,\Lambda)$. Hence the locus of wobbly bundles is a closed subset 
$$\Ww \subset \M_X(2,\Lambda).$$
It was announced in Laumon \cite{L} Remarque 3.6 (ii) that $\Ww$ is of pure codimension $1$. 
The term ``wobbly" was introduced in the paper \cite{DP}.

\bigskip
Our first result proves this claim. Since the isomorphism class of the moduli space $\M_X(2, \Lambda)$ depends only
on the parity of the degree $\lambda = \deg \Lambda$, it will be enough to study two cases, $\lambda = 0$ and $\lambda =1$.
For $1 \leq k \leq g- \lambda$ we define $\Ww_k$ to be the 
closure in $\M_X(2, \Lambda)$ of the locus of all semi-stable vector bundles arising
as extensions
\[
0 \longrightarrow L \longrightarrow E \longrightarrow \Lambda L^{-1} \longrightarrow 0,
\]
with $\deg L = 1-k$ and $\dim H^0 (X, KL^2 \Lambda^{-1}) > 0$. We denote by $\lc x \rc$ the ceiling 
of the real number $x$. With this notation we have the following results.

\begin{theorem}
The wobbly locus $\Ww \subset \M_X(2, \Lambda)$ is of pure codimension $1$ and we have the following 
decomposition for $\lambda = 0$ and $\lambda = 1$
$$ \Ww = \Ww_{\lc \frac{g-\lambda}{2} \rc} \cup \ldots \cup \Ww_{g-\lambda}.$$
In particular, all loci $\Ww_k$ appearing in the above decomposition are divisors. They are all irreducible, except
$\Ww_g$ for $\lambda = 0$, which is the union of $2^{2g}$ irreducible divisors.
\end{theorem}

This theorem completes the results obtained in \cite{P} showing that $\Ww$ is of codimension $1$ for $\lambda = 1$.
The idea of the proof is to consider the rational 
forgetful map (forgetting the non-zero Higgs field) from the equidimensional nilpotent cone in the moduli space
of semi-stable Higgs bundles to $\M_X(2, \Lambda)$. It turns out that roughly half of the irreducible components
of the nilpotent cone gets contracted by the forgetful map with one-dimesional fibers to the above mentioned
divisors $\Ww_k$, and the other half gets contracted with fibers of dimension $> 1$ to subvarieties of these disisors
$\Ww_k$.

\bigskip

Our second result studies the relationship between very stable vector bundles $E$ and the loci of maximal line 
subbundles of $E$. We recall here the main results on maximal line subbundles of rank-$2$ bundles (see e.g. 
\cite{O}, \cite{LN}). Under the assumption that $g+ \lambda$ is odd, i.e., $g$ odd if $\lambda = 0$ and
$g$ even if $\lambda=1$, the Quot-scheme
$$ M(E) := \mathrm{Quot}^{1, 1 - \lc \frac{g}{2} \rc}(E)  $$
parameterizing subsheaves of rank $1$ and degree $1 - \lc \frac{g}{2} \rc$ of $E$
is a zero-dimensional, reduced scheme of length $2^g$ for a {\em general} bundle
$E \in \M_X(2, \Lambda)$. In that case $M(E)$ parameterizes line subbundles of $E$ of maximal degree. We say that
$M(E)$ is non-degenerate if $\dim M(E) = 0$ and $M(E)$ is reduced, and degenerate if the opposite holds.

\begin{theorem}
Under the assumption that $g+ \lambda$ is odd, the following holds.
\begin{enumerate}
    \item If $E$ is very stable, then $M(E)$ is non-degenerate.
    \item The subscheme $M(E)$ is degenerate for any $E \in \Ww_{\lc \frac{g - \lambda}{2} \rc}$.
\end{enumerate}
\end{theorem}

The case $g=2, \lambda =1$ was already worked out in \cite{P}. In Remark 3.1 we show that part (2) does not hold on other components
of the wobbly locus.

\bigskip

Our last result computes the class $cl(\Ww_k)$ of the wobbly divisors $\Ww_k$ in the Picard group of the 
moduli space $\M_X(2, \Lambda)$, which is isomorphic to $\mathbb{Z}$ (see e.g. \cite{DN}).

\begin{theorem}
We have the following equality for $\lambda = 0$ and $\lambda = 1$
$$ cl(\Ww_k) = 2^{2k} \binom{g}{2g-2k - \lambda} \ \text{for} \  \lc \frac{g- \lambda}{2} \rc \leq k \leq g- \lambda.$$
\end{theorem}

In the case $\lambda = 0$ the computations of the class $cl(\Ww_k)$ were already carried out in \cite{F}, but due to some
typos the final result in loc.cit. is not correct. For the convenience of the reader we include a detailed 
presentation of the computations in the case $\lambda=1$. 

\bigskip

In the last section we give a description of these divisors for low genus.

\bigskip

We would like to thank Jochen Heinloth and Ana Pe\'on-Nieto for useful discussions on the nilpotent cone. 
The first author thanks the University of Nice Sophia-Antipolis for financial support of a visit in November 2017, when 
most of this work was carried out.

\section{Proof of Theorem 1.1}

Let $\text{Higgs}_X(2, \Lambda)$ be the moduli space of semi-stable Higgs bundles of rank $2$ with fixed determinant $\Lambda$. 
The Hitchin map defined by mapping a Higgs field $(E, \phi)$ to its determinant $\det(\phi)$ is a proper surjective map 
\[
h: \text{Higgs}_X(2, \Lambda) \to H^0(X, K^2).
\] 
It is easy to see that the nilpotent cone decomposes as 
$$h^{-1}(0) = M_X(2, \Lambda) \cup \tilde{\Ee},$$
where $\M_X(2, \Lambda)$ denotes here  pairs $(E, 0)$ with zero Higgs field and $\tilde{\Ee}$ consists of semi-stable pairs $(E, \phi)$ with non-zero nilpotent 
Higgs field --- note that the underlying bundle $E$ is not necessarily semi-stable. 
In other words, the image of $\tilde{\Ee}$ under the rational forgetful map $h^{-1}(0) \dashrightarrow M_X(2, \Lambda)$ is the locus of wobbly
bundles. In the case $\lambda = 1$  the nilpotent cone was already described in \cite{TD}. For the convenience of the
reader we recall now the description.

\bigskip
By \cite{P} Lemma 3.1 a vector bundle $E$ admits a non-zero nilpotent Higgs field 
if and only if it contains a line subbundle $L$ with $H^0(X, KL^2\Lambda^{-1}) \ne 0$.
Thus any such bundle can be written as an extension 
\begin{equation} \label{extension}
   0 \to L \xrightarrow{i} E \xrightarrow{\pi} L^{-1}\Lambda \to 0,  
\end{equation}

 where $\Lambda$ is a line bundle of degree $\lambda \in \{0, 1\}$ and the nilpotent Higgs field is given as the composition 
 $\phi = i \circ u \circ \pi$ with a non-zero $u \in H^0(X, KL^2\Lambda^{-1}) = \mathrm{Hom}(L^{-1} \Lambda, LK)$.
 Then we have the following inequalities : \\

 $\bullet$
 Since   $\phi(L) = 0, L$ is invariant under $\phi$ and by semi-stability of the pair $(E, \phi)$, we have 
 $\text{deg}(L) = d \le \frac{\lambda}{2}$. \\
 
 $\bullet$
 We also have $u \ne 0$. This implies that $\deg(KL^2\Lambda^{-1}) \ge 0 \Leftrightarrow d \ge \frac{\lambda}{2} + 1 -g.$\\

 Hence we obtain the inequalities 
 $$\frac{\lambda}{2} +1 -g \le d \le \frac{\lambda}{2}.$$
 We set $k= 1-d$ and we distinguish two cases: \\
 
 $\bullet$  $\lambda = 0: 1-g \le d \le 0 \Leftrightarrow 1 \le k \le g$.\\
 
 $\bullet$  $\lambda =1: \frac{1}{2} +1 -g \le d \le \frac{1}{2} \Leftrightarrow 1 \le k \le g-1.$\\
 
 We introduce the subloci 
 $$
 \begin{array}{rcl}
 \Ww^0_k & := & \{ E \in \Ww: E \  \text{contains a line subbundle} \  L \text{ of degree 1-k} \\
       &    & \text{with } H^0(X, KL^2\Lambda^{-1}) \ne 0 \}
 \end{array}
 $$
 and denote by $\Ww_k$ the Zariski closure of $\Ww_k^0$ in $\M_X(2, \Lambda)$.
 We therefore deduce from the above considerations the following decompositions
 $\Ww = \bigcup_{k=1}^g \Ww_k$ for $\lambda = 0$ and $\Ww = \bigcup_{k=1}^{g-1} \Ww_k$ for $\lambda = 1$. \\
 
 \begin{remark} \label{Kummervariety}
 We observe that for $\lambda = 0$ the locus $\Ww_1$ coincides with the semi-stable boundary of $\M_X(2, \Lambda)$, which equals the
 Kummer variety of $X$.
 \end{remark}

 Now we decompose $\tilde{\Ee}$ as $\bigcup_{k=1}^g \Ee_k$ for $\lambda =0$ and $\bigcup_{k=1}^{g-1} \Ee_k$ for $\lambda = 1$ 
 such that the image of $\Ee_k$ under the forgetful map is $\Ww_k$.
 The construction goes as follows (we omit the construction of $\Ee_1$ for $\lambda = 0$ --- see Remark \ref{Kummervariety}) :\\
 We introduce the subvarieties $Z_k \subset \text{Pic}^{1-k}(X)$ for $1 \le k \le g-\lambda$ defined by 
 \[
 Z_k:= \{L \in \text{Pic}^{1-k}(X) \text{ such that } h^0(X, KL^2\Lambda^{-1}) \ne 0 \}.
 \]
 Then one can construct $Z_k$ as the pre-image of the Brill-Noether locus $W_{2g-2k-\lambda}(X)$ under the map 
 \[
 \mu_k: \text{Pic}^{1-k}(X) \to \text{Pic}^{2g-2k-\lambda}(X)
 \]
 taking $L$ to $KL^2\Lambda^{-1}.$ Then 
 $$
     \begin{array} {rcll}
     \text{dim} \  Z_k & = & 2g-2k-\lambda & \text{ if } 2g-2k-\lambda \le g \\
                       & = &  g  & \text{ if } 2g-2k-\lambda \ge g.
\end{array}
$$
Note that in the latter case $Z_k = \text{Pic}^{1-k}(X)$. Consider the fiber product $\tilde{Z}_k
= Z_k \times_{W_{2g-2k-\lambda}} S^{2g-2k-\lambda}(X)$
 \begin{equation}\label{E1}
 \xymatrix{
 \tilde{Z}_k \ar[r] \ar[d]^q  & S^{2g-2k-\lambda}(X) \ar[d]\\
 Z_k \ar[r]^{\mu_k} &  W_{2g-2k-\lambda}
 } 
 \end{equation}
 where the right vertical map is the natural map from the symmetric product of the curve to its Picard variety. Then the projection map 
 $p: \tilde{Z}_k \to S^{2g-2k-\lambda}(X)$ is a $2^{2g}$-fold \'etale covering of $S^{2g -2k-\lambda}(X)$ and $\tilde{Z}_k$ 
 parameterizes line bundles $L$ and effective divisors in $|KL^2\Lambda^{-1}|$.
 There exists a unique line bundle $\mathcal{L}$ over $S^{2g-2k-\lambda}(X)$ whose fiber at a divisor $D$
 is canonically isomorphic to the space of sections of the line bundle determined
by $D$ which vanish precisely on $D$. Furthermore, it can be shown that this line bundle is trivial.

\bigskip
Excluding the case $\lambda = 0$ and $k=1$ (see Remark \ref{Kummervariety}), we observe that the dimension of $\mathrm{Ext}^1(\Lambda L^{-1}, L) = H^1(\Lambda^{-1} L^2)$
depends only on the degree of the line bundle $L$. Therefore there exists a vector bundle
$\Vv_k$ over $Z_k$ whose fiber at a point $L \in Z_k$  is canonically isomorphic to
$\mathrm{Ext}^1(\Lambda L^{-1}, L)$. The rank of the vector bundle $\Vv_k$ is $g +2k +\lambda -3$ and a 
general extension class $v$ in the fiber $(\Vv_k)_L$ defines a 
stable rank-$2$ vector bundle $E_v$.

\begin{proposition} \label{nilpcone}
We have the following :
\begin{enumerate}
    \item The total space of the vector bundle $q^* \Vv_k \oplus p^* \mathcal{L}$ over $\tilde{Z}_k$ parameterizes triples
    $(L,v,u)$, where $L$ is a line bundle in $Z_k$, $v$ is an extension class in the fiber $(\Vv_k)_L$ and $u$ is a global
    section of $KL^2 \Lambda^{-1}$.
    \item There exists a rational map $\phi_k$ from the projectivized bundle $\PP(q^* \Vv_k \oplus p^* \mathcal{L})$ to $\mathrm{Higgs}_X(2, \Lambda)$
    defined by sending $(L,v,u)$ to the Higgs bundle $(E_v, i \circ u \circ \pi)$ as defined by the exact sequence (\ref{extension}).
    \item We have a commutative diagram
    \begin{equation}
 \xymatrix{
 \PP(q^* \Vv_k \oplus p^* \mathcal{L}) \ar[r]^{\phi_k} \ar[d]  & \mathrm{Higgs}_X(2,\Lambda) \ar[d]\\
 \PP( \Vv_k ) \ar[r]^{\psi_k} &  \M_X(2, \Lambda),
 } 
 \end{equation}
where all arrows are rational maps. The vertical maps are forgetful maps of the global section of $KL^2 \Lambda^{-1}$ and of
the Higgs field respectively.
   \item The restriction of $\phi_k$ to the vector bundle $q^* \Vv_k \subset \PP(q^* \Vv_k \oplus p^* \mathcal{L})$ is an injective morphism.
\end{enumerate}
\end{proposition}

\begin{proof}
Part (1) follows immediately from the previous description of $\Vv_k$ and $\mathcal{L}$. As for part (2) it will be 
enough to show that the Higgs bundle associated to the triple $(L, \lambda v, \lambda u)$ does not depend on the scalar 
$\lambda \in \CC^*$. But this follows from the observation that the extension class of the exact sequence obtained from
(\ref{extension}) by replacing either $i$ by $\frac{1}{\lambda} i$ or $\pi$ by $\frac{1}{\lambda} \pi$ equals $\lambda v \in
\mathrm{Ext}^1(\Lambda L^{-1}, L)$ for any $\lambda \in \CC^*$. Part (3) and part (4) are straightforward.
\end{proof}

Now we define $\Ee^0_k$ to be the image of the rational map $\phi_k$ and $\Ee_k$ its Zariski closure. 
Then clearly $\Ee_k \subset \Ee$ and 
$\Ww^0_k$ is the image of $\Ee^0_k$ under the forgetful map. 
Clearly $\Ee_k$ is irreducible, except if $\dim Z_k =0$ which 
is equivalent to $\lambda =0$ and $k=g$, and since $\phi_k$ is generically injective
$$
\begin{array}{rcl}
\dim \Ee_k & = & \dim \tilde{Z}_k + \mathrm{rk} \ \Vv_k \\
           & = & (2g -2k-\lambda) + (g + 2k + \lambda -3) \\
           & = & 3g-3. 
\end{array}
$$
We note that the fiber over a general element $E \in \Ww^0_k$ of the forgetful map $\Ee^0_k \rightarrow \Ww^0_k$
is $H^0(X, KL^2\Lambda^{-1})$, where $L$ is a line bundle of degree $-k+1$ contained in $E$. If $2g-2k-\lambda \le g$ and $L$ is a general line bundle of degree $-k+1$ with $H^0(X, KL^2\Lambda^{-1}) \ne 0$, then $h^0(X, KL^2\Lambda^{-1})=1$. Therefore $\text{dim} \Ww^0_k = 3g-3-1= 3g-4$ for $2g-2k-\lambda \le g \Leftrightarrow  \frac{g-\lambda}{2} \le k.$ Thus 
$\mathcal{W}_k$ is an irreducible  divisor in $M_X(2, \Lambda)$ for $k\ge \frac{g-\lambda}{2}$ and $k \not= g$.
 
\begin{proposition}
We have the inclusions
$$\bigcup_{k=1}^{\lc \frac{g-\lambda}{2} \rc -1} \Ww_k \subset 
\Ww_{\lc \frac{g-\lambda}{2}\rc}.$$
\end{proposition}

\begin{proof}
We put $k_0 =  \lc \frac{g-\lambda}{2} \rc$ and consider the rational map  
\[
\psi_{k_0} : \mathbb{P}(\Vv_{k_0}) \dashrightarrow \M_X(2, \Lambda)
\]
introduced in Proposition \ref{nilpcone} (3). Let $L \in Z_{k_0}$.  The restriction of 
$\psi_{k_0}$ on the fiber of $\mathbb{P}(\Vv_{k_0})_L$ over $L$ is not defined at the points where the associated bundles are not 
semi-stable. Let
\[
\psi_L: \mathbb{P}_L := \mathbb{P}(H^1(X, L^2\Lambda^{-1})) \dashrightarrow \M_X(2, \Lambda)
\]
be the restriction of $\psi_{k_0}$ at the fiber over $L$.
Then by \cite{B} Theorem 1 there is a natural sequence $\sigma$ of blow-ups along smooth centers resolving $\psi_L$  
into a morphism $\tilde{\psi}_L: \tilde{\mathbb{P}}_L \to M_X(2, \Lambda)$.  The image of $\tilde{\psi}_L$ is contained in closure
of the image of the rational map $\psi_L$. Here $\sigma$ is the blow-up morphism $\sigma: \tilde{\mathbb{P}}_L \to \mathbb{P}_L$. 

Now $X$ is embedded in $\mathbb{P}_L$ via the natural map. Let $x \in X$ and 
 $E$ be the bundle associated to $x$. Then $E$ fits into the exact sequence (\ref{extension}) and 
by \cite{B} Observation (2) page 451 the bundle $E$ is not semi-stable. 
Furthermore by \cite{B} Theorem 1 (2) there is a natural isomorphism $\sigma^{-1}(x) \cong 
\tilde{\mathbb{P}}_{L(x)}$ and, when restricted to $\sigma^{-1}(x)$, the morphism 
$\tilde{\psi}_L$ coincides with the morphism
\[
\tilde{\psi}_{L(x)} : \tilde{\mathbb{P}}_{L(x)}  \to  \M_X(2, \Lambda).
\]
Now $\tilde{\mathbb{P}}_{L(x)}$ is the blow-up of $\mathbb{P}_{L(x)}$ and the bundles corresponding to extension classes in 
$\mathbb{P}_{L(x)}$ fit in the exact sequence of the form
\[
0 \to L(x) \to V \to L^{-1}(-x) \Lambda \to 0.
\]
Hence we deduce that the image of $\tilde{\psi}_{L(x)}$ is contained in $\Ww_{k_0}$. Next we observe that if 
$L \in Z_{k_0}$, then $L(x) \in Z_{k_0-1}$ for any $x \in X$. Hence we obtain a morphism
\[
\mu : Z_{k_0} \times X \longrightarrow Z_{k_0 -1}, \ \  \mu(L,x) = L(x).
\]
It will be enough to show that $\mu$ is surjective to conclude that $\Ww_{k_0-1}^0 \subset \Ww_{k_0}$, hence $\Ww_{k_0-1} \subset \Ww_{k_0}$ since
$\Ww_{k_0}$ is closed.

If $g- \lambda$ is even, then $Z_{k_0-1} = \mathrm{Pic}^{2-k_0}(X)$ et $Z_{k_0} = \mathrm{Pic}^{1-k_0}(X)$ and $\mu$ is obviously surjective. 
If $g - \lambda$ is odd, then $Z_{k_0-1} = \mathrm{Pic}^{2-k_0}(X)$ and $Z_{k_0}$ is an irreducible divisor in $\mathrm{Pic}^{1-k_0}(X)$. If on the contrary $\mu$ is not surjective, then 
$Z_{k_0}$ would be
invariant by a translation by an line bundle of the form $\mathcal{O}_X(x-y)$ for $x,y \in X$, hence by any translation. This is a contradiction.

\bigskip

More generally let $D$ be a general effective divisor of degree $1 \le d \le k_0 -1$. Then again by \cite{B} Observation 2 any point  $x \in \overline{D}$ corresponds to a 
non-semi-stable bundle, except if $\lambda = 0$ and $d = k_0 -1$ (see below for a discussion of this exceptional case). 
Furthermore if $x$ is general in $\overline{D}$ then by \cite{B} Theorem 1 (2) there is a natural isomorphism  $\sigma^{-1}(x) \cong \tilde{\mathbb{P}}_{L(D)}$ and the restriction of 
$\tilde{\psi}_L$ to $\sigma^{-1}(x)$ coincides with the map
\[
\tilde{\psi}_{L(D)} : \tilde{\mathbb{P}}_{L(D)}  \to M_X(2, \Lambda).
\]
As before, the natural multiplication morphism
$$ \mu : Z_{k_0} \times S^{d}(X) \longrightarrow Z_{k_0 -d}$$
is easily seen to be surjective, which implies that $\Ww_{k_0-d}^0 \subset \Ww_{k_0}$, hence $\Ww_{k_0-d} \subset \Ww_{k_0}$ since
$\Ww_{k_0}$ is closed.

Finally, the case $\lambda = 0$ and $d= k_0 -1$ corresponds to $\Ww_1$, which is the image of the $(k_0 -1)$-th secant variety to $X \subset
\mathbb{P}_L$ under the rational map $\psi_L$, when $L$ varies in $Z_{k_0}$.
\end{proof}

\bigskip

We will need the following lemma in section 4.

\begin{lemma} \label{fiberqzerodim}
With the notation of Proposition \ref{nilpcone} and for $\lc \frac{g-\lambda}{2}\rc \leq k \leq g - \lambda$, a {\em general} line bundle 
$L \in Z_k$ and a {\em general} extension class $v \in (\Vv_k)_L = \mathrm{Ext}^1(\Lambda L^{-1}, L)$ defining a bundle $E_v$ as in 
(\ref{extension}), we have
$$ \dim \mathrm{Hom}(E_v, L^{-1} \Lambda) = 1. $$
\end{lemma}

\begin{proof}
Since $L$ is general in $Z_k$ we have $h^0(KL^2 \Lambda^{-1}) = 1$, or equivalently by Riemann-Roch and Serre duality
$h^1(KL^2\Lambda^{-1}) = h^0(L^{-2} \Lambda) = \lambda + 2k -g$. Applying the functor $\mathrm{Hom}( - , L^{-1} \Lambda)$
to the short exact sequence  (\ref{extension}) we see that $\dim \mathrm{Hom}(E_v, L^{-1} \Lambda) = 1$ if and only if the 
coboundary map given by the cup product $\cup v$  with the extension class $v$
\[
\cup v : H^0(L^{-2} \Lambda) = \mathrm{Hom}(L, L^{-1} \Lambda) \longrightarrow H^1 (\mathcal{O}_X) = \mathrm{Ext}^1(L^{-1} \Lambda, L^{-1} \Lambda)
\]
is injective. Given a non-zero section $s \in H^0(L^2 \Lambda)$ it is well-known that $s \cup v = 0$ if and only if 
the extension class $v \in \mathrm{Ext}^1( \Lambda L^{-1}, L) = H^1(L^2 \Lambda^{-1}) = H^0(K L^{-2} \Lambda)^*$
lies in the linear span $\langle D \rangle \subset |K L^{-2} \Lambda|^*$, where $D$ is the zero divisor of $s$. But
$\dim \langle D \rangle  = 2k - 4 + \lambda$, so 
\[
\dim \bigcup_{D \in |L^{-2} \Lambda|} \langle D \rangle \leq (\lambda + 2k -g-1) + (2k -4 + \lambda)= 4k -g+2 \lambda -5,   
\]
which is $< \dim \mathbb{P} (\Vv_k)_L = g + 2k + \lambda - 4$. So for a general extension class $v$ we see that $s \cup v \not= 0$
for any non-zero $s \in H^0(L^{-2} \Lambda)$, which is equivalent to  $\dim \mathrm{Hom}(E_v, L^{-1} \Lambda) = 1$.
\end{proof}

\bigskip

\section{Proof of Theorem 1.2}

We only consider the case when $\lambda = 1$ and $g$ is even, i.e., $g=2a$ for some integer $a$. 
The proof in the other case can be carried out similarly. If $g=2a$, then $\lc \frac{g-1}{2} \rc = a$. 
In this situation, a general rank-$2$ vector bundle of degree $1$ has a line subbundle of maximal degree $1-a$ (see e.g.\cite{O}, \cite{LN}). 

\bigskip
{\em Proof of (1):}
Let $E$ be a very stable rank-$2$ vector bundle of degree $1$. Suppose on the contrary 
that $\text{dim} \ M(E) > 0$ or $M(E)$ is non-reduced. Let $L_0 \in M(E)$.  If $L_0 \to E$ is not saturated, then we have a sequence of maps
\[
L_0 \to L \to E \to L^{-1}\Lambda \to L_0^{-1}\Lambda, 
\]
where $\text{deg} L \ge \text{deg} L_0 +1 = 2-a$. Then $\chi(L^{-2}\Lambda) \le -2$. Therefore 
$h^1(X, L^{-2} \Lambda) = h^0(X, KL^2\Lambda^{-1}) >0$, which implies that $E$ contains a line subbundle $L$ with 
$h^0(X, KL^2\Lambda^{-1}) \ne 0$. Then by \cite{P} Lemma 3.1 the bundle $E$ is not very stable, a contradiction.

On the other hand, if $L_0 \to E$ is saturated, then $E$ fits in the exact sequence
\[
0 \to L_0 \to E \to L_0^{-1}\Lambda \to 0.
\]
Note that the Zariski tangent space at $L_0$ is given by $\text{Hom}(L_0, L_0^{-1}\Lambda)$.  Therefore if $\text{dim} \ M(E) \ge 1$ or 
if $L_0$ is a non-reduced point in $M(E)$, then $\dim \text{Hom}(L_0, L_0^{-1}\Lambda) >0$.
But $\chi(L_0^{-2}\Lambda) = 0$. Thus if $\dim \mathrm{Hom}(L_0, L_0^{-1}\Lambda) >0$, then $h^1(X, L_0^{-2} \Lambda) = h^0(X, KL^2\Lambda) >0$. 
Therefore $E$ is not very stable, a contradiction.

\bigskip
{\em Proof of (2):}
Let $E \in \mathcal{W}^0_a$. Then $E$ contains a line subbundle $L$ of degree $1-a$ such that $h^0(X,KL^2\Lambda^{-1}) \ne 0$ and we 
have the exact sequence (\ref{extension}). Since $\chi(L^{-2}\Lambda) = 0$, we obtain 
that $h^0(X, L^{-2}\Lambda) = h^0(X, KL^2\Lambda^{-1}) >0$. Therefore the dimension of the tangent space at $L$ to the 
Quot-scheme $M(E)$ is $h^0(X, L^{-2}\Lambda) > 0$. Therefore $M(E)$ is degenerate at $L$. Finally, since being degenerate is a closed condition, $M(E)$ is degenerate for
any $E \in \Ww_a$.

\bigskip

\begin{remark}
The statement in Theorem 1.2 (2) is not valid for the points in the other components $\mathcal{W}_k$ for $k >\lc \frac{g-\lambda}{2} \rc$.

\bigskip

{\bf{Example}:}  Take $\lambda = 0$ and $g = 2a+1=3$, i.e.,  $a =1$. For simplicity we assume that the curve $X$ is non-hyperelliptic. 
Then the wobbly locus has two components, 
$\mathcal{W}_2 \text{ and } \mathcal{W}_3$, where $\mathcal{W}_2$ is an irreducible divisor, but $\mathcal{W}_3$ is 
a union of $64$ hyperplane sections.  The $64$ hyperplane sections are indexed by the $64$ theta-characteristics $\theta$,
i.e. line bundles satisfying $\theta^2= K$. We claim that a general extension class 
$\xi \in \mathbb{P}(H^1(X, \theta^{-2})) \dashrightarrow \mathcal{W}_3 \subset M_X(2, \mathcal{O})$ is such that $M(E_\xi)$ is non-degenerate. 
Note that $\mathbb{P}(H^1(X, \theta^{-2})) = \mathbb{P}^5= \mathbb{P}({H^0(X, K^2)}^*)= {|K^2|}^*$ contains the following secant varieties to the curve 
$X \hookrightarrow {|K^2|}^*$
\[
X= \text{Sec}^1(X) \hookrightarrow {|K^2|}^*, \ \ \text{Sec}^2(X) \hookrightarrow {|K^2|}^*,
\]
which are of dimension 1 and 3 respectively, and we have an equality (see e.g. \cite{Lan}) 
\begin{equation}\label{A1}
\text{Sec}^3(X) = {|K^2|}^*.
\end{equation}
It is well-known, see e.g. \cite{LN} Proposition 1.1, that an extension class $\xi \in {|K^2|}^*$ lies on $\text{Sec}^i(X)$ for $1 \le i \le 3$ if and only if the 
associated vector bundle $E_{\xi}$ contains a line subbundle of degree $2-i$. Let $\xi \in \text{Sec}^3(X) \setminus \text{Sec}^2(X)$. Then 
the maximal degree of line subbundles of $E_{\xi}$ is $-1$. Let $D= x_1 + x_2 + x_3$ be a general effective divisor on $X$ of degree $3$ such that 
$\xi \in \langle D \rangle$, the linear span of three points $x_1, x_2, x_3$ of $D$ in ${|K^2|}^*$. 
Then $E_{\xi}$ contains the line subbundle $\theta(-D)$. Note that $E_{\xi}$ is an extension 
\[
0 \to \theta^{-1} \to E_{\xi} \to \theta \to 0.
\]
We also note that $L =\theta(-D)$ is a reduced point of the Quot-scheme $M(E_{\xi})$ if and only if 
$\text{Hom}(L, {E_{\xi}}/L) = \text{Hom}(L, L^{-1}) = H^0(X, L^{-2}) = \{ 0  \}$.

Now consider the map
\[
\Phi: S^{3}(X) \to \text{Pic}^2(X)
\]
which takes a point $(x_1, x_2, x_3)$ to $\mathcal{O}(2x_1 + 2x_2 +2x_3) \otimes \theta^{-2}$.
Clearly the map $\Phi$ is surjective. Let $Z$ denote the divisor $\Phi^{-1}(\Theta)$, where $\Theta$ 
denotes the theta divisor in $\text{Pic}^2(X)$. Then $D \in Z$ if and only if 
$h^0(X, \mathcal{O}(2D) \otimes \theta^{-2}) \ne 0$. We denote by $\tilde{Z}$ the span of all
projective planes $\langle D \rangle \subset {|K^2|}^*$ when $D$ varies in $Z$.
Then $\tilde{Z}$ is a divisor in ${|K^2|}^*$ and we have $\xi \in \tilde{Z}$ if and only if 
$E_{\xi}$ contains a line subbundle $L$ such that $H^0(X, L^{-2}) \ne \{ 0 \}$.  
So for general $\xi \notin \tilde{Z}$ the set $M(E_{\xi})$ is reduced and consists of $8$ line subbundles.

\bigskip

A similar computation can be done for $\lambda =1$ and $g=4$ to show that the statement in Theorem 1.2 (2) is 
not valid for the points in the component  $\mathcal{W}_3$.

\end{remark}

\section{Proof of Theorem 1.3}

In this section we compute the class $cl(\Ww_k)$ of the wobbly divisor $\Ww_k$ in the case $\lambda = 1$ 
for $\lc \frac{g-1}{2} \rc \leq k \leq g-1$ following closely the method used in \cite{F} Section 5 Example 1. Note
that in \cite{F} the case $\lambda = 0$ is worked out.

\bigskip

Let $S$ be a smooth connected variety and let $\Ee$ be a rank-$2$ vector bundle over $S \times X$ such that $\det \Ee =
\pi_X^* (\Lambda)$, where $\pi_S$ and $\pi_X$ denote the projections onto $S$ and $X$ respectively, and such that 
$\Ee_s := \Ee_{|\{ s\} \times X}$ is stable for any $s \in S$. Then the family $\Ee$ determines a classifying map
$$ f : S \longrightarrow \M_X(2, \Lambda).$$

Our first task is to compute the first Chern class of the pull-back under $f$ of the ample generator
$\Dd$ of the Picard group of $\M_X(2, \Lambda)$, i.e.,
$$ \Theta_S := c_1( f^* \Dd) \in H^2(S),$$
in terms of Chern classes of $\Ee$. We recall \cite{DN} Th\'eor\`eme B that the line bundle $f^* \Dd$ is defined as the inverse
of the determinant line bundle
$$ \det R\pi_{S*}(\Ee \otimes \pi_X^* H),$$
where $H$ is a rank-$2$ vector bundle of degree $2g-3$. Note that the condition on the degree 
is equivalent to $\chi (\Ee_s \otimes H) = 0$. Then the Grothendieck-Riemann-Roch theorem gives the equalities
$$
\begin{array}{rcl}
\Theta_S & = & -c_1 ( \det R\pi_{S*}(\Ee \otimes \pi_X^* H)) \\
         & = & -\frac{1}{2} \pi_{S*} \left[ c_1 (\Ee \otimes \pi_X^* H)^2 -2c_2(\Ee \otimes \pi_X^* H) -c_1(\Ee \otimes \pi_X^* H)\cdot \pi_X^* (c_1(K))  \right] \\
         & = & \pi_{S*} c_2(\Ee \otimes \pi_X^* H) \\
         & = & 2\pi_{S*} c_2(\Ee) \in H^2 (S).
\end{array}
$$
Since we need to compute the class in $H^2(S)$ of the $k$-th wobbly divisor $f^{-1}(\Ww_k) \subset S$ in terms of $\Theta_S$, it will be enough 
to do the computations modulo classes in $H^i(S)$ for $i \geq 3$. Hence the above relation allows to write 
$$ c_2(\Ee) = \frac{1}{2} \Theta_S \otimes \eta  \in H^4(S \times X),$$
where $\eta \in H^2(X)$ denotes the class of a point in $X$ --- note that we omit classes in $H^4(S) \otimes H^0(X)$ and $H^3(S) \otimes H^1(X)$.
Since $c_1(\Ee) = 1 \otimes \eta  \in H^0(S) \otimes H^2 (X) \subset H^2(S \times X)$ we get the following expression for the Chern character of $\Ee$
$$ ch(\Ee) = 2 + 1 \otimes \eta - \frac{1}{2} \Theta_S \otimes \eta  + h.o.t., $$
where all h.o.t. are contained in $\oplus_{i \geq 3} H^i(S) \otimes H^*(X)$. 

\bigskip
We also need to recall some standard facts on the first Chern class of a Poincar\'e bundle $\Ll$ over 
$P \times X$, with $P := \mathrm{Pic}^{1-k}(X)$ 
(see e.g. \cite{ACGH} page 335). We have
$$c_1(\Ll) = (1-k) 1 \otimes \eta + \gamma \in H^2(P \times X),$$
where $\gamma$ denotes a class in $H^1(P) \otimes H^1(X)$ --- for a more precise description of $\gamma$ see \cite{ACGH} --- with the
property 
$$\gamma^2 = -2 \Theta_P \otimes \eta.$$
Here $\Theta_P \in H^2(P)$ denotes the class of a theta divisor in $P$.
The rest of the computations goes exactly as in the case $\lambda = 0$. For the convenience of the reader we include the details.

\bigskip

The main idea is to realize the $k$-th wobbly divisor
$$ f^{-1}(\Ww_k) = \pi_S (\Delta_k \cap (S \times Z_k)) \subset S$$
as the projection onto $S$ of the intersection of $S \times Z_k$ with the determinantal
subvariety $\Delta_k \subset S \times P$ defined by 
\[
\Delta_k = \{(s,L) \in S \times P \ | \ \mathrm{Hom}(\Ee_s, L^{-1} \Lambda ) \not= 0 \},
\]
and which is constructed by the standard technique as follows. We fix a reduced divisor $D_0$ of degree 
$d_0$ sufficiently large such that $h^1(X, Hom(\Ee_s, L^{-1} \Lambda (D_0))) = 0$ 
for all $s \in S$ and $L \in P$. We consider the exact sequence over $S \times P \times X$
\[
0 \rightarrow Hom(\Ee, \Ll^{-1} \Lambda) \rightarrow Hom(\Ee, \Ll^{-1} \Lambda(D_0)) \rightarrow 
Hom(\Ee, \Ll^{-1} \Lambda(D_0))_{|D_0} \rightarrow 0.
\]
We introduce the following two vector bundles over $S \times P$
\[
\Ff := (\pi_{S \times P})_* \left( Hom(\Ee, \Ll^{-1} \Lambda(D_0))\right) \ \text{and} \ 
\Aa := \bigoplus_{x \in D_0} Hom(\Ee, \Ll^{-1})_{| S \times P \times \{ x \}} 
\]
of ranks $2(d_0 + k - g) + 1$ and $2d_0$ respectively. Taking the direct image of the above exact sequence under the projection
$\pi_{S \times P}$ onto $S \times P$ we obtain a map $\phi: \Ff \rightarrow \Aa$ over $S \times P$. Let us denote by $q : \mathbb{P}(\Ff) 
\rightarrow S \times P$ the projection from the projectivized bundle $\mathbb{P}(\Ff)$ onto the base variety $S \times P$. Then the 
composition of the tautological section over $\mathbb{P}(\Ff)$ with $q^* \phi$
\[
\mathcal{O}(-1) \rightarrow q^* \Ff \rightarrow q^* \Aa
\]
defines a global section $s \in H^0 (\mathbb{P}(\Ff), q^* \Aa \otimes \mathcal{O}(1))$ whose zero set
 equals
\[
\tilde{\Delta}_k = \{ (s,L, \overline{\varphi}) \ | \ (s,L) \in \Delta_k \ \text{and} \ \overline{\varphi} 
\in \mathbb{P}(\mathrm{Hom}(\Ee_s, L^{-1} \Lambda)) \}.
\]
By Lemma \ref{fiberqzerodim} the map $\tilde{\Delta}_k \cap q^{-1} (S \times Z_k) \rightarrow \Delta_k \cap (S \times Z_k)$ induced by the projection $q$
is birational, which implies that 
$$\dim \tilde{\Delta}_k \cap q^{-1} (S \times Z_k) = \dim \Delta_k \cap (S \times Z_k) = \dim S - 1.$$ 
Hence 
$$\dim \tilde{\Delta}_k  \leq \dim S -1 + \mathrm{codim} \  Z_k = \dim \mathbb{P}(\Ff)  - 2d_0,$$
which shows that $\mathrm{codim} \ \tilde{\Delta}_k  = 2d_0$. Hence we can conclude that its fundamental class is given by the top
Chern class
$$ [\tilde{\Delta}_k] = c_{2d_0}(q^* \Aa \otimes \mathcal{O}(1)).$$
Moreover, by the projection formula we have
$$
\begin{array}{rcl}
[\Delta_k] & = & q_* [\tilde{\Delta}_k] \\
           & = & \sum_{i = 0}^{2d_0} q_*\left( c_i(q^* \Aa) c_1(\mathcal{O}(1))^{2d_0 - i} \right) \\
           & = & \sum_{i = 0}^{2d_0} c_i(\Aa) q_* \left(c_1(\mathcal{O}(1))^{2d_0 - i} \right) \\
           & = & q_* \left(c_1(\mathcal{O}(1))^{2d_0} \right)  \ \text{mod} \ H^i(S) \ i\geq 3.
\end{array}
$$
The last equality follows from the facts that
\[
c_1(Hom(\Ee, \Ll^{-1})_{| S \times P \times \{ x \}}) = 0  \ \text{and} \ c_2(Hom(\Ee, \Ll^{-1})_{| S \times P \times \{ x \}})
\in H^4(S) \otimes H^0(P),
\]
which imply that 
\[
c_1(\Aa) = 0 \ \text{and} \ c_k(\Aa) \in H^{2k}(S) \otimes H^0(P) \ \text{for} \ k \geq 2.
\]
In order to compute the class $q_* \left(c_1(\mathcal{O}(1))^{2d_0} \right)$ we  
compute by the Grothendieck-Riemann-Roch theorem the first terms of the
Chern character of $\Ff$.
$$
\begin{array}{rcl}
ch(\Ff) & = & (\pi_{S \times P})_* ( ch(\Ee^*) \cdot ch(\Ll^{-1} \Lambda(D_0)) \cdot Td(X)  ) \\
        & = & (\pi_{S \times P})_*  ( (2 - 1 \otimes \eta - \frac{1}{2} \Theta_S \otimes \eta + h.o.t) \cdot
        (1 + (d_0 + k) 1 \otimes \eta + \gamma  \\ 
        &   & - \Theta_P \otimes \eta ) \cdot ( 1 + (1 - g) 1 \otimes \eta) ) \\
        & = &  (\pi_{S \times P})_*  ( 2 + (2d_0 + 2k -2 g + 1) 1 \otimes \eta + 2 \gamma - 2 \Theta_P \otimes \eta
        - \frac{1}{2} \Theta_S \otimes \eta \\
        &   & + h.o.t. ) \\
        & = &  r - \frac{1}{2} \Theta_S  - 2 \Theta_P  + h.o.t. ,
\end{array}
$$
where all h.o.t. of the last line are contained in 
$\oplus_{i \geq 3} H^i(S) \otimes H^*(P)$ and 
$r = 2(d_0 + k - g) +1$ denotes the rank of $\Ff$. So we obtain
$$c_1(\Ff) = - (\frac{1}{2} \Theta_S  + 2 \Theta_P)$$
and working modulo the ideal $\oplus_{i \geq 3} H^i(S) \otimes H^*(P)$ we easily 
show the following relations for $k \geq 1$
$$ c_k (\Ff) = \frac{1}{k!} (c_1(\Ff))^k  \ \text{mod} \ H^i(S) \ i\geq 3.$$
Hence we can write the Chern polynomial of $\Ff$ as
$$c_t(\Ff) = 1 + c_1(\Ff) t + \cdots + c_r(\Ff) t^r = 
\exp{(c_1(\Ff)t)}.$$
The class $q_* \left(c_1(\mathcal{O}(1))^{2d_0} \right)$ is by definition (see e.g. \cite{Fu} Chapter 3) the $(2d_0 - r + 1)$-th Segre class of $\Ff$ and is computed as
the coefficient of $t^{2d_0 - r + 1}$ of the inverse of the Chern polynomial
$$ c_t(\Ff)^{-1} = \exp{(-c_1(\Ff)t)},$$
which equals
$$ \frac{(\frac{1}{2} \Theta_S  + 2 \Theta_P)^{e+1}}{(e+1)!},$$
where we put $e = 2d_0 - r = 2g -2k -1 = \dim Z_k$. We only need to compute the component
in $H^2(S) \otimes H^{2e}(P)$ of this class, which equals
$$ \frac{(e+1) \frac{1}{2} \Theta_S \otimes  2^e \Theta_P^e}{(e+1)!} =
\Theta_S \otimes \frac{2^{e-1}}{e!} \Theta_P^e.$$
In order to conclude we will need the following fact.
\begin{lemma}
The fundamental class of $Z_k$ in $P$ equals
$$ [Z_k] = \frac{2^{2(g-e)}}{(g-e)!} \Theta_P^{g-e}. $$
\end{lemma}

\begin{proof}
We recall that the duplication map of an abelian variety $A$ acts as multiplication by $2^n$ on the
cohomology $H^n(A, \mathbb{C})$. We apply this fact to the map $\mu_k$ defining $Z_k$ and we obtain
$$ [Z_k] = \mu_k^*[W_e(X)] = \frac{2^{2(g-e)}}{(g-e)!} \Theta_P^{g-e}, $$
where $W_e(X) \subset \Pic^e(X)$ denotes the Brill-Noether locus of line bundles $L$  with
$h^0(L) > 0$, whose fundamental class equals $\frac{\Theta_P^{g-e}}{(g-e)!}$
by Poincar\'e's formula.
\end{proof}

We now combine the previous results and we obtain
$$ [\Delta_k][Z_k] = \Theta_S \otimes \frac{2^{e-1} 2^{2(g-e)}}{e! (g-e)!} \Theta^g_P = \Theta_S 2^{2g-e-1}\binom{g}{e}
= \Theta_S 2^{2k} \binom{g}{2g-2k -1}, $$
which gives the class $cl(\Ww_k)$ stated in Theorem 1.3.

\bigskip

\begin{remark}
We observe that in \cite{F} page 350 the factor $\binom{g}{e}$ is missing in the formula giving 
``the integral of $\Theta^e_J$ over the preimage of $C^e$".
\end{remark}

\section{Examples}

\subsection{Genus $2$}

\subsubsection{$\lambda = 0$} It is known that $\M_X(2, \mathcal{O}_X)$ is isomorphic to $\mathbb{P}^3$. By Theorem 1.1 the wobbly locus has two components
$$\mathcal{W}= \mathcal{W}_1 \cup \mathcal{W}_2,$$ 
where  $\mathcal{W}_k$ is the closure of the locus $\mathcal{W}_k^0$ for $k =1, 2$.\\ 
Let $k =1$. Note that for any line bundle $L$ of degree zero $h^0(X, KL^2) \ne 0$ and any bundle which contains a line subbundle of degree zero is semi-stable.
Therefore, $\mathcal{W}_1$ is precisely the locus of semi-stable bundles which are not stable. It is known that the strictly semi-stable locus
is a quartic hypersurface (known as Kummer surface) in $\mathbb{P}^3$. Thus the class $\text{cl}(\mathcal{W}_1)$ of the wobbly divisor
$\mathcal{W}_1$ in the Picard group of $\M_X(2, \mathcal{O}_X)$ is $4 \Theta$, where $\Theta$ is the ample generator of the Picard group of $\M_X(2, \mathcal{O}_X)$.

Let $k =2$. Then for a line bundle $L$, $h^0(X, KL^2) \ne 0$ if and only if $L$ is the inverse of a theta characteristic. There are precisely $16$ such line bundles. 
If $L$ is such a line bundle, then any nontrivial extension of $L$ by $L^{-1}$ is stable and for each such line bundle $L$ the space of extensions 
gives a hyperplane in $\mathbb{P}^3$. Therefore $\mathcal{W}_2$ is the union of $16$ hyperplanes in $\mathbb{P}^3$ and the class $\text{cl}(\mathcal{W}_2)$ of the wobbly divisor
$\mathcal{W}_2$ in the Picard group of $\M_X(2, \mathcal{O}_X)$ is $16 \Theta$.

\subsubsection{$\lambda= 1$}  Let $\Lambda$ be a line bundle of degree $1$. It is known that $\M_X(2, \Lambda)$ is isomorphic to a smooth intersection  $Y$ of two 
quadrics in $\mathbb{P}^5.$ By Theorem 1.1 the wobbly locus is irreducible.
If a stable vector bundle $E$ is in the wobbly locus, then under the identification of $\M_X(2, \Lambda)$ with $Y$, it 
corresponds to a point $P \in Y$ such that the intersection of $Y$ with the projectivized 
embedded tangent space of $Y$ at $P$ contains fewer than $4$ lines \cite{P}. 
Classically, it is known that the locus of such points $P \in Y$ is isomorphic to a surface in $\mathbb{P}^5$ of degree $32$. In other words, the irreducible wobbly divisor is isomorphic to a surface in $\mathbb{P}^5$ of degree $32$.
Thus the class $\text{cl}(\mathcal{W}_1)$ of the wobbly divisor
$\mathcal{W}_1$ in the Picard group of $\M_X(2, \Lambda)$ is $8 \Theta$, where $\Theta$ is a hyperplane section (of degree $4$) of $\M_X(2, \Lambda)$.

\subsection{Genus $3, \lambda = 0, k= 2$}  It is known that $\M_X(2, \mathcal{O}_X)$ is isomorphic to Coble's quartic hypersurface in $\mathbb{P}^7$ \cite{NR}. On the other hand, by Theorem 1.3 we have that the class $\text{cl}(\mathcal{W}_2)$ of the wobbly divisor $\mathcal{W}_2$ in the Picard group
 of $\M_X(2, \mathcal{O}_X)$ is $48 \Theta$, where $\Theta$ is a hyperplane section of degree $4$. Therefore, we can describe  $\mathcal{W}_2$ as the cut out of Coble's quartic by a hypersurface of degree $48$.

\subsection{Arbitrary genus $g$, $\lambda = 0$, $k = g$} We recall that $\mathcal{W}_g^0 = \{E: E $ contains a line subbundle $L$ of degree $1-g$  
 with $h^0(KL^2) \ne 0 \}$.  For a line bundle $L$ of degree $1-g$, $h^0(KL^2) \ne 0$ if and only if $L$ is the inverse of a theta characteristic. For each such 
 line bundle $L$ the space of non-trivial extension classes of $L$ by $L^{-1}$ gives a divisor of $\M_X(2, \mathcal{O}_X)$, whose class is
 the ample generator $\Theta$ of the Picard group of $\M_X(2, \mathcal{O}_X)$. Therefore the $2^{2g}$ irreducible divisors of $\mathcal{W}_g$ correspond 
 to the $2^{2g}$ theta characteristics of $X$. Thus the class $\text{cl}(\mathcal{W}_g)$ of the wobbly divisor $\mathcal{W}_g$ in the Picard group
 of $\M_X(2, \mathcal{O}_X)$ is $2^{2g} \Theta$.

\end{document}